\documentclass[12pt, a4paper]{amsart}
\usepackage{times,amsmath,url,amscd}
\usepackage{amssymb}
\usepackage{layout}
\usepackage{amsthm}
\usepackage[all]{xy}

\numberwithin{equation}{section}

\newtheorem{thm}[equation]{Theorem}
\newtheorem{prop}[equation]{Proposition}

\newtheorem{lem}[equation]{Lemma}
\theoremstyle{definition}
\newtheorem{dfn}[equation]{Definition}

\theoremstyle{plain}
\newtheorem{cor}[equation]{Corollary}

\theoremstyle{remark}
\newtheorem*{remark}{Remark}

\DeclareMathOperator{\Hom}{Hom}

\DeclareMathOperator{\Ker}{Ker}
\DeclareMathOperator{\im}{Im}
\DeclareMathOperator{\Aut}{Aut}
\DeclareMathOperator{\Ext}{Ext}

\DeclareMathOperator{\End}{End}

\newcommand{\N}{\mathbb{N}}
\newcommand{\Z}{\mathbb{Z}}

\newcommand{\Q}{\mathbb{Q}}

\renewcommand{\phi}{\varphi}
\newcommand{\A}{\mathcal{A}}

\newcommand{\K}{\mathcal{K}}

\renewcommand{\geq}{\geqslant}

\let\oldmod\mod
\renewcommand{\mod}{\!\!\oldmod}

\begin{document}

\title[Upper triangular matrices and operations in $K$-theory]{Upper triangular matrices and operations in odd primary connective $K$-theory}

\author{Laura Stanley}

\author{Sarah Whitehouse}
\address{School of Mathematics and Statistics, University of Sheffield, Sheffield S3 7RH, UK.}
\email{s.whitehouse@sheffield.ac.uk}

\begin{abstract}
We prove analogues for odd primes of results of Snaith and Barker-Snaith.
Let $\ell$ denote the $p$-complete connective Adams summand and consider 
the group of left $\ell$-module automorphisms of $\ell\wedge\ell$ in the stable homotopy category 
which induce the identity on mod $p$ homology.
We prove a group isomorphism between this group and a certain group
of infinite invertible upper triangular matrices with entries in the $p$-adic integers. We 
determine information about the matrix corresponding to the automorphism $1\wedge\Psi^q$ of $\ell\wedge\ell$,
where $\Psi^q$ is the Adams operation and $q$ is an integer which generates the $p$-adic units.
\end{abstract}

  
\date{$18^{\text{th}}$ April 2012}
\maketitle

\section{Introduction}
\label{SecIntro}

We prove analogues for odd primes of results of Snaith and Barker-Snaith~\cite{Snaith-utt, BaSn}.
Let $\ell$ denote the Adams summand of the $p$-complete connective $K$-theory spectrum and let
$\Aut_{\text{left-}\ell\text{-mod}}^0(\ell\wedge\ell)$ be the group of left $\ell$-module automorphisms 
 of $\ell\wedge\ell$ in the stable homotopy category which induce the identity on mod $p$ homology.
The first main result is Theorem~\ref{iso}, which gives a group isomorphism between this group and a certain group
of infinite invertible upper triangular matrices with entries in the $p$-adic integers. The second main result is
Theorem~\ref{isomatrix}, which determines information about the matrix corresponding to the automorphism $1\wedge\Psi^q$ of $\ell\wedge\ell$,
where $\Psi^q$ is the Adams operation and $q$ is an integer which generates the $p$-adic units $\Z_p^\times$.
An application is given to the important maps $1\wedge\phi_n$ where $\phi_n=(\Psi^q-1)(\Psi^q-\hat{q})\dots(\Psi^q-\hat{q}^{n-1})$ 
and $\hat{q}=q^{p-1}$.
\medskip

While the general strategy of the proofs is the same as in the $2$-primary case, there
are differences in algebraic and combinatorial details. In places we obtain entirely new information.
Notably, Theorem~\ref{app} gives a new closed formula, involving $q$-binomial coefficients, for
each entry in the matrix corresponding to the map $1\wedge\phi_n$.
\smallskip

This article is organized as follows. Section 2 contains the proof of the 
upper triangular matrix result. Section 3 presents an explicit basis for the
torsion-free part of $\pi_*(\ell\wedge \ell)$. This is used in Section 4 to
obtain information about the matrix corresponding to $1\wedge \Psi^q$. Applications
are discussed in Section 5 and there is a short appendix about the $q$-binomial theorem. 
\smallskip

This paper is based on work in the Ph.D. thesis of the first author~\cite{stanley},
produced under the supervision of the second author.

\section{Upper triangular technology}
\label{SecUTT}

In this section we prove the odd primary analogue of a theorem of Snaith~\cite[Theorem 1.2]{Snaith-utt}; see also~\cite[Theorem $3.1.2$]{Vic'sBook}. This provides an identification between a group of 
$p$-adic infinite upper triangular matrices and certain operations for the Adams summand of complex connective $K$-theory.

Let $ku$ be the $p$-adic connective complex $K$-theory spectrum and let $\ell$ be the $p$-adic Adams summand. 

\begin{dfn}
Let $\End_{\text{left-}\ell\text{-mod}}(\ell\wedge\ell)$ be the ring of left $\ell$-module endomorphisms of $\ell\wedge\ell$ of degree zero
in the stable homotopy category
and let $\Aut_{\text{left-}\ell\text{-mod}}(\ell\wedge\ell)$ be the group of units of this ring. Denote by $\Aut^0_{\text{left-}\ell\text{-mod}}(\ell\wedge\ell)$ the subgroup consisting of those homotopy equivalences which induce the identity map in mod $p$ homology.
\end{dfn}

\begin{dfn}\label{matrix group}
Consider the group (under matrix multiplication) of invertible infinite upper triangular matrices with entries in the $p$-adic integers $\Z_p$.
An element is a matrix $X=(X_{i,j})$ for $i,j\in\N_0$, where $X_{i,j}\in\Z_p$, $X_{i,j}=0$ for $i>j$,
and $X_{i,i}\in \Z_p^\times$.
Let $U_\infty\Z_p$ be the subgroup with all diagonal entries lying in the subgroup $1+p\Z_p$ of $\Z_p^\times$.
\end{dfn}

The main theorem of this section is as follows.

\begin{thm}\label{iso}
There is an isomorphism of groups 
    $$
    \Lambda:U_\infty\Z_p\xrightarrow{\cong}\Aut^0_{\textup{left-}\ell\textup{-mod}}(\ell\wedge \ell).
    $$
\end{thm}

The method of proof is essentially that used by Snaith to prove the analogous result for $p=2$, but there are differences of detail.
A basic ingredient in the proof is the following splitting 
    $$
    \ell\wedge\ell\simeq\ell\wedge\bigvee_{n\geqslant0}\K(n).
    $$ 
This goes back to work of Kane~\cite{Kane}. (He claimed the result in the $p$-local setting. In~\cite{cdgm} a gap in his argument was identified    
 and fixed in the $p$-complete situation.) The spectra $\K(n)$ appearing in the splitting are suspensions of Brown-Gitler spectra,
realising a weight filtration of the homology of $\Omega^2S^3\langle 3\rangle_p$. They are $p$-complete finite spectra.
(We remark that in the $2$-primary case studied by
Snaith, the pieces of the splitting should also be suspensions of Brown-Gitler spectra, 
rather than the finite complexes $F_{4n}/F_{4n-1}$.)

The splitting means that it is enough to study left $\ell$-module maps of the form         
$\phi_{m,n}:\ell\wedge\K(m)\rightarrow \ell\wedge\K(n)$ 
for each $m$, $n\geqslant0$.
We use a suitable Adams spectral sequence to identify particular maps 
$\iota_{m,n}:\ell\wedge\K(m)\rightarrow\ell\wedge\K(n)$ 
which are represented by generators of certain groups on the $E_2$ page of the spectral sequence. These 
maps $\iota_{m,n}$ are used to define the required isomorphism.

The work in this section is organised as follows. Firstly, 
we set up the required Adams spectral sequence. Next 
we establish the stable isomorphism class of
the mod $p$ cohomology of $\K(n)$ in order to simplify the $E_2$ term.
We note that the spectral sequence collapses at the $E_2$ term. We then pick generators of certain groups on the $E_2$ page 
to give the maps $\iota_{m,n}$ used in the definition of the map. The spectral sequence is then further analysed to show that
this map is bijective. Finally we show that the choice of the maps $\iota_{m,n}$ can be made in such a way that $\Lambda$ is a group isomorphism.
\medskip

Since we consider left $\ell$-module maps, a map $\phi_{m,n}$ as above is determined by its restriction to $S^0\wedge\K(m)\rightarrow \ell\wedge\K(n)$. This is an element of the homotopy group 
    $$
    [\K(m),\ell\wedge\K(n)]=[S^0,\ell\wedge\K(n)\wedge D(\K(m))]_p.
    $$
Here we are abusing notation slightly by writing $D(\K(m))$ to mean the $p$-completion of the Spanier-Whitehead dual of the finite spectrum $Y$,
where $\K(m)\simeq Y_p$.

We write $\A_p$ for the mod $p$ Steenrod algebra and we let $B=E[Q_0, Q_1]\subset \A_p$ be the exterior 
subalgebra generated by $Q_0=\beta$ and $Q_1$, where $Q_0$ has degree $1$ and $Q_1$ has degree $2p-1$.

All homology and cohomology groups will be with coefficients in $\Z/p$ unless explicitly stated otherwise; we omit the coefficients from
the notation.

We use the Adams spectral sequence with $E_2$ term 
    \begin{equation}
    \label{ASS}
    E_2^{s,t}=\Ext_{\A_p}^{s,t}(H^*(\ell \wedge \K(n) \wedge D(\K(m))),\Z/p)
    \end{equation}
and which converges to
    $$
    E^{s,t}_\infty=[S^0,\ell \wedge \K(n) \wedge D(\K(m))]_{t-s}\otimes\Z_p
        =\pi_{t-s}(\ell \wedge \K(n) \wedge D(\K(m)))\otimes\Z_p.
    $$

\noindent The $E^2$ term simplifies in a standard way as follows.
    \begin{align}
        E_2^{s,t}&=\Ext_{\A_p}^{s,t}(H^*(\ell \wedge \K(n) \wedge D(\K(m))),\Z/p)\notag\\
                &\cong\Ext_{\A_p}^{s,t}(H^*(\ell)\otimes H^*(\K(n))\otimes H^*(D(\K(m))),\Z/p)\notag\\
                &\cong\Ext_{\A_p}^{s,t}((\A_p\otimes_B\Z/p) \otimes H^*(\K(n))\otimes H^*(D(\K(m))),\Z/p)\notag\\
                &\cong\Ext_{\A_p}^{s,t}(\A_p\otimes_B (H^*(\K(n))\otimes H^*(D(\K(m)))),\Z/p)\notag\\
                &\cong\Ext_B^{s,t}(H^*(\K(n)) \otimes H^*(D(\K(m))) ,\Z/p).\notag
    \end{align}
The first two isomorphisms are by the K\"{u}nneth theorem and the fact that $H^*(\ell)\cong\A_p\otimes_B\Z/p$ - see
~\cite[Part III, Proposition $16.6$]{AdamsSH}, respectively.
Next we use the isomorphism of left $\A_p$-modules 
    $$
    (\A_p\otimes_B\Z/p)\otimes M\cong\A_p\otimes_B M
    $$ 
where $\A_p$ acts diagonally on the 
left-hand side by the comultiplication and on the right-hand side by multiplication within $\A_p$ -
see~\cite[Part III, Proof of Proposition $16.1$]{AdamsSH}. Finally we use a standard change of rings isomorphism.

To simplify the $E^2$ term further, we use the theory of stable isomorphism classes; see
~\cite[Part III, Chapter 16]{AdamsSH}.
Stable isomorphism of modules over $B$ will be denoted by $\approxeq$. 
The stable classes we need are expressible in terms of two basic
$B$-modules, the augmentation ideal $I$ of $B$ and the $B$-module $\Sigma$ with a single copy of $\Z/p$ in degree $1$.
We denote the $a$-fold tensor power of $I$ by $I^a$, and similarly for $\Sigma$.

We write $\nu_p$ for the $p$-adic valuation function.

\begin{thm}\label{stable Kn}
There are stable isomorphisms  
    \begin{align*}
    H^*(\K(n))&\approxeq\Sigma^{2n(p-1)-\nu_p(n!)}I^{\nu_p(n!)},\\
    H^*(D(\K(n)))&\approxeq\Sigma^{\nu_p(n!)-2n(p-1)}I^{-\nu_p(n!)}.
    \end{align*}
\end{thm}

\begin{proof}
From~\cite[Lemma 8:3, Lemma 8:4]{Kane}, we have the following calculations
of the $Q_0$ and $Q_1$ homology of $H^*(\K(n))$, each of which is
concentrated in a single degree.
    \begin{align*}
        H(H^*(\K(n));Q_0)=\Z/p &\text{\ \  in dimension $2n(p-1)$ and}\\
        H(H^*(\K(n));Q_1)=\Z/p &\text{\ \  in dimension $2(p-1)(\nu_p(n!)+n)$}.
    \end{align*}
The first stable isomorphism then follows from~\cite[Part III, Theorem $16.3 $]{AdamsSH}.

The Universal Coefficient Theorem gives us the $B$-module isomorphism 
    $$
    H^*(\K(n))\cong\Hom^*_{\Z/p}(H_{-*}(\K(n)),\Z/p).
    $$ 
For any invertible $B$-module, its linear dual is its inverse stable isomorphism class, by~\cite[Part III, Lemma 16.3(i)]{AdamsSH},
so it follows from the above that 
    $$
    H_{-*}(\K(n))\approxeq\Sigma^{\nu_p(n!)-2n(p-1)}I^{-\nu_p(n!)}.
    $$ 
Then Spanier-Whitehead duality gives us the $B$-module isomorphism 
    $$
    H^*(D(\K(n)))\cong H_{-*}(\K(n)),
    $$ 
which gives the second stable isomorphism.
\end{proof}

\begin{cor}
In the spectral sequence~(\ref{ASS}) we have for $s>0$,
    $$
    E^{s,t}_2\cong\Ext_B^{s+\nu_p(n!)-\nu_p(m!),t-2(n-m)(p-1)+\nu_p(n!)-\nu_p(m!)}(\Z/p,\Z/p).
    $$
\end{cor}

\begin{proof}
This follows from Theorem~\ref{stable Kn} and 
the standard dimension-shifting isomorphisms of $\Ext$ groups
    \begin{equation}\label{dimshift}
    \begin{aligned}
    \Ext^{s,t}_B(I\otimes M,\Z/p)&\cong\Ext^{s+1,t}_B(M,\Z/p)\\
    \Ext^{s,t}_B(\Sigma\otimes M,\Z/p)&\cong\Ext^{s,t-1}_B(M,\Z/p)
    \end{aligned}
    \end{equation}
for $s>0$ and $M$ a $B$-module.
\end{proof}

\begin{lem}
The spectral sequence~(\ref{ASS}) collapses at the $E_2$ term.
\end{lem}

\begin{proof}
Recall that 
$\Ext^{*,*}_B(\Z/p,\Z/p)=\Z/p[c,d]$ where $c\in\Ext_B^{1,1}$ and $d\in\Ext_B^{1,2p-1}$.
It follows from the above that, away from the line $s=0$, all non-zero terms are in even total degrees,
so there are no non-trivial differentials when $s>0$. 

Showing there are no non-trivial differentials when $s=0$ can be done by the method of~\cite[Part III, Lemma 17.12]{AdamsSH}.
Consider an element $e\in E_2^{0,t}$ where $t$ is odd (if $t$ is even, there can be no non-trivial differentials for degree reasons). 
We proceed by induction. Suppose that $d_i=0$ for $i<r$, so $E_2^{s,t}\cong E_r^{s,t}$. 
We have $ce=0$ as this lies in odd total degree, hence $d_r(ce)=0$.
But $cd_r(e)=d_r(ce)$ because this is a spectral sequence of modules over $\Ext_B^{*,*}(\Z/p,\Z/p)$,
so $cd_r(e)=0$. Away from the $s=0$ line, the $E_2=E_r$ page of the spectral sequence reduces to a polynomial algebra 
with $c$ corresponding to one of the generators, so multiplication by $c$ is a monomorphism on $E_r^{s,t}$ for $s>0$. Thus 
$d_r(e)=0$ which completes the induction.
\end{proof}

\begin{dfn}\label{iota}
For $m\geq n$, 
let ${\iota_{m,n}}:\ell\wedge\K(m)\rightarrow \ell\wedge\K(n)$ be a map which is represented in the spectral sequence by a choice of generator of         
    $$
    E_2^{m-n-\nu_p(n!)+\nu_p(m!),m-n-\nu_p(n!)+\nu_p(m!)}.
    $$
Also let $\iota_{m,m}$ be the identity on $\ell\wedge\K(m)$.
\end{dfn}

\begin{prop}\label{lambda bij}
There is a bijective map 
    $$
    \Lambda:U_\infty\Z_p\rightarrow\Aut^0_{\text{left-}\ell\text{-mod}}(\ell\wedge \ell),$$
given by
    $$
    X\mapsto\sum_{m\geqslant n}X_{n,m}\iota_{m,n}:\ell\wedge(\bigvee_{i\geqslant0}\K(i))\rightarrow \ell\wedge(\bigvee_{i\geqslant0}\K(i)).
    $$
\end{prop}

\begin{proof}
Firstly, we check that $\Lambda$ does have the correct target.

Note that a left $\ell$-module endomorphism of $\ell\wedge\ell$ which induces the identity
on mod $p$ homology corresponds to a collection of maps $\phi_{m,n}: \ell\wedge \K(m)\rightarrow \ell\wedge \K(n)$,
where $\phi_{m,n}$ induces the zero map for $m\neq n$ and each $\phi_{m,m}$ induces the identity.
In the spectral sequence, elements represented in the $s=0$ line are detected in mod $p$ homology.
So, for $m\neq n$, we are interested in elements of 
$\pi_0(\ell\wedge\K(n)\wedge D(\K(m)))\otimes \Z_p$ 
represented in the spectral sequence in $E_2^{s,s}=E_\infty^{s,s}$
with $s>0$.

We have, $E_2^{s,s}=\Ext^{u,v}_B(\Z/p,\Z/p)$ for $s>0$, where
    \begin{align*}
        u&=s+\nu_p(n!)-\nu_p(m!),\\
        v&=s-2(n-m)(p-1)+\nu_p(n!)-\nu_p(m!).
    \end{align*}
So $v-u=2(m-n)(p-1)$. If $n>m$ then $u>v$ and these groups are all zero.
This explains the range $m\geqslant n$ in the definition of the map $\Lambda$.

It is clear that $\sum_{m\geqslant n}X_{n,m}\iota_{m,n}$  defines a left-$\ell$-module endomorphism of $\ell\wedge\ell$. 
It is easy to check that it is invertible because the coefficient $X_{m,m}$ of each identity map $\iota_{m,m}$ is a unit. For $m\neq n$,
we have chosen $\iota_{m,n}$ represented away from the $s=0$ line, so this map
induces the zero map on mod $p$ homology. Evidently $\iota_{m,m}$ induces the identity map, and since its coefficient
lies in $1+p\Z_p$ the resulting map on $\ell\wedge\ell$ induces the identity on 
mod $p$ homology. Hence $\Lambda$ does take values in $\Aut^0_{\text{left-}\ell\text{-mod}}(\ell\wedge \ell)$.

To show that $\Lambda$ is bijective, we consider non-trivial homotopy classes of left-$\ell$-module maps of the form         
$\phi_{m,n}:\ell\wedge\K(m)\rightarrow \ell\wedge\K(n)$,
where $m\geqslant n$, such that $\phi_{m,n}$
induces the identity on mod $p$ homology if $m=n$ and induces zero if $m>n$.

We start with $m>n$. As we have seen, a map $\phi_{m,n}$ as above is
represented in the spectral sequence in $E_2^{s,s}=E_\infty^{s,s}$ with $s>0$. 
Any non-zero $E_2^{s,s}=\Ext_B^{u,v}(\Z/p,\Z/p)$ group is isomorphic to $\Z/p$ generated by 
    $$
    c^{s+n-m+\nu_p(n!)-\nu_p(m!)}d^{m-n}.
    $$ 
This group is non-zero precisely when $s\geqslant(m-n)-\nu_p(n!)+\nu_p(m!)$.
So the map $\phi_{m,n}$ is represented in 
    $$
    E_\infty^{j+m-n-\nu_p(n!)+\nu_p(m!),j+m-n-\nu_p(n!)+\nu_p(m!)}
    $$ 
for some integer $j\geqslant 0$.

\noindent If
    $$
    E_\infty^{m-n-\nu_p(n!)+\nu_p(m!),m-n-\nu_p(n!)+\nu_p(m!)}=\Z/p\{x\}
    $$ 
then 
    $$
    E_\infty^{j+m-n-\nu_p(n!)+\nu_p(m!),j+m-n-\nu_p(n!)+\nu_p(m!)}=\Z/p\{c^jx\}.
    $$
The ring structure of the spectral sequence yields that multiplication by $c$ in the spectral sequence 
corresponds to multiplication by $p$ on $\pi_0(\ell \wedge \K(n) \wedge D(\K(m)))\otimes\Z_p$, so we see 
that 
    $$
    \phi_{m,n}=\gamma p^j\iota_{m,n}
    $$ 
for some $p$-adic unit $\gamma$ and integer $j\geqslant 0$.

If $m=n$, we need to consider the terms $E_2^{s,s}=E_\infty^{s,s}$ for $s\geqslant 0$. 
Here we see that $E_2^{s,s}=E_\infty^{s,s}=\Z/p\{c^s\}$ for $s\geqslant 0$. Again multiplication
by $c$ corresponds to multiplication by $p$ on $\pi_0(\ell \wedge \K(n) \wedge D(\K(n)))\otimes\Z_p$.
Thus
    $$
    \phi_{m,m}=\gamma p^j\iota_{m,m}
    $$ 
for some $p$-adic unit $\gamma$ and integer $j\geqslant 0$. The map $\phi_{m,m}$ induces the identity on mod $p$ homology if and only if $j=0$ and $\gamma\in 1+p\Z_p$, corresponding to the condition that the diagonal entries of the matrix lie in $1+p\Z_p$.

This shows that, for each collection of maps $\phi_{m,n}$ corresponding to an element of the target $\Aut^0_{\text{left-}\ell\text{-mod}}(\ell\wedge \ell)$, there is a unique choice of $X_{n,m}\in\Z_p$ for $m> n$ and $X_{m,m}\in 1+p\Z_p$, such that the map
is the image under $\Lambda$ of the matrix $X$.
\end{proof}

We now fix choices of the maps $\iota_{m,n}$ in such a way that $\Lambda$ is a group isomorphism.

\begin{prop}\label{extprod}
We can choose the maps $\iota_{m,n}$ as follows.
As before let $\iota_{m,m}$ be the identity map on $\ell\wedge\K(m)$, let $\iota_{m+1,m}$ be as already described, then let     
    $$
    \iota_{m,n}=\iota_{n+1,n}\iota_{n+2,n+1}\cdots\iota_{m,m-1}
    $$ 
for all $m>n+1$. Then 
    \begin{displaymath}
        \iota_{m,n}\iota_{k,l}=\begin{cases}
                    \iota_{k,n} & \textrm{ if } k\geqslant l=m\geqslant n,\\
                    0 & \textrm{ otherwise,}\end{cases}
    \end{displaymath}
and with these choices $\Lambda$ is an isomorphism of groups.
\end{prop}

\begin{proof}
Let $\iota_{m,n}$ be any choice of generator as in Definition~\ref{iota}. To justify that these can be chosen as above,
we need to consider the relationship between the product $\iota_{m,n}\iota_{k,m}$ and $\iota_{k,n}$,
for $k>m>n$. Let $s(m,n)=m-n-\nu_p(n!)+\nu_p(m!)$, then $\iota_{m,n}$ is represented by a generator of     
    $$
    \Ext^{s(m,n),s(m,n)}_B(\Sigma^{2(n-m)(p-1)+\nu_p(m!)-\nu_p(n!)}I^{\nu_p(n!)-\nu_p(m!)},\Z/p),
    $$ 
$\iota_{k,m}$ is represented by a generator of 
    $$
    \Ext^{s(k,m),s(k,m)}_B(\Sigma^{2(m-k)(p-1)+\nu_p(k!)-\nu_p(m!)}I^{\nu_p(m!)-\nu_p(k!)},\Z/p)
    $$ 
and $\iota_{k,n}$ is represented by a generator of 
    $$
    \Ext^{s(k,n),s(k,n)}_B(\Sigma^{2(n-k)(p-1)+\nu_p(k!)-\nu_p(n!)}I^{\nu_p(n!)-\nu_p(k!)},\Z/p).
    $$ 

The product $\iota_{m,n}\iota_{k,m}$ is represented by the product of the representatives under the pairing of $\Ext$ groups     
    $$
    \Ext^{s,s}(\Sigma^aI^b,\Z/p)\otimes \Ext^{s',s'}(\Sigma^{a'}I^{b'},\Z/p)\rightarrow\Ext^{s+s',s+s'}(\Sigma^{a+a'}I^{b+b'},\Z/p)
    $$ 
induced by the isomorphism $\Sigma^aI^b\otimes\Sigma^{a'}I^{b'}\cong\Sigma^{a+a'}I^{b+b'}$. We can identify this pairing using the following 
commutative diagram.
\footnotesize\begin{equation*}
\xymatrix{
    \Ext^{s,s}(\Sigma^aI^b,\Z/p)\otimes\Ext^{s',s'}(\Sigma^{a'}I^{b'},\Z/p) \ar[r] \ar[d]_{\cong} & \Ext^{s+s',s+s'}(\Sigma^{a+a'}I^{b+b'},\Z/p) \\
    \Ext^{s+b,s-a}(\Z/p,\Z/p)\otimes\Ext^{s'+b',s'-a'}(\Z/p,\Z/p) \ar[r] & \Ext^{s+s'+b+b',s+s'-a-a'}(\Z/p,\Z/p) \ar[u]_{\cong} }
\end{equation*}
\normalsize
The bottom pairing is the Yoneda splicing and it is an isomorphism when all the groups are non-zero as any non-zero $\Ext$ group here is a copy of $\Z/p$. The vertical isomorphisms are the dimension-shifting isomorphisms. So the top pairing is an isomorphism whenever the groups are non-zero
and since $s(k,m)+s(m,n)=s(k,n)$ this holds in our case. Hence up to a $p$-adic unit $u_{k,m,n}$ we have     
    $$
    \iota_{m,n}\iota_{k,m}=u_{k,m,n}\iota_{k,n},
    $$ 
and we can choose the maps $\iota_{m,n}$ as stated above.

Now $\Lambda$ is a group isomorphism because
    \begin{align*}
        \Lambda(X)\Lambda(Y)&=\left(\sum_{m\geqslant n}X_{n,m}\iota_{m,n}\right)\left(\sum_{k\geqslant l}Y_{l,k}\iota_{k,l}\right)
                                =\sum_{k\geqslant l=m\geqslant n}X_{n,m}Y_{l,k}\iota_{m,n}\iota_{k,l}\\
                            &=\sum_{k\geqslant l\geqslant n}X_{n,l}Y_{l,k}\iota_{k,n}
                                =\sum_{k\geqslant n}(XY)_{n,k}\iota_{k,n}\\
                            &=\Lambda(XY).\qedhere
\end{align*}
\end{proof}

\noindent Hence we have now proved Theorem~\ref{iso}.

\section{A basis of the torsion-free part of $\pi_*(\ell\wedge\ell)$}
\label{Secbasis}

In this section, we find a basis for the torsion-free part of the homotopy groups $\pi_*(\ell\wedge\ell)$. 
To do this we follow methods introduced by Adams in~\cite{AdamsSH}. We then study some of the properties of 
this basis including how it relates to Kane's splitting. We explore its behaviour with 
relation to the Adams spectral sequence in order to assess the effect of the maps $(\iota_{m,n})_*$. 
This will allow us to compare with the effect of $(1\wedge\Psi^q)_*$ 
and hence, in the next section, to deduce information about the matrix corresponding to $1\wedge\Psi^q$ 
under the isomorphism $\Lambda$.
\smallskip

It would be interesting to compare the basis that we find here with elements 
of the torsion free part of $\pi_*(\ell\wedge\ell)$
studied in~\cite[\S9,10]{br}. We hope to return to this in future work.

\subsection{A basis}
\label{subsec:basis}

We consider the torsion-free part of $\pi_*(\ell\wedge\ell)$ by considering its image in 
$\pi_*(\ell\wedge\ell)\otimes\Q_p=\Q_p[\hat{u},\hat{v}]$,
where $\pi_*(\ell)=\Z_p[\hat{u}]$. 

We fix a choice of $q$ primitive modulo $p^2$, so that $q$ is a topological 
generator of the $p$-adic units $\Z_p^\times$ and we let $\hat{q}=q^{p-1}$.
We also adopt the notation $\rho=2(p-1)$.

The integrality conditions governing the image can be found as in~\cite[Part III, Theorem 17.5]{AdamsSH}
and are as follows.

\begin{prop}\label{subring}
For $f(\hat{u},\hat{v})\in\Q_p[\hat{u},\hat{v}]$ to be in the image 
of $\pi_*(\ell\wedge\ell)$ it is necessary and sufficient for $f$ to satisfy the following two conditions.
\newcounter{itemcounter}
\begin{list}
{(\arabic{itemcounter})}
{\usecounter{itemcounter}\leftmargin=0.5em}
\item $f(kt,lt)\in\Z_p[t]$ for all $k$, $l\in 1+p\Z_p$.
\item $f(\hat{u},\hat{v})$ is in the subring $\Z_p[\frac{\hat{u}}{p},\frac{\hat{v}}{p}]$.\qed
\end{list}
\end{prop}

We begin with the following polynomials. They have been chosen by starting from the basis elements for
$L_0(l)$ given in~\cite[Proposition 4.2]{CCW2} and multiplying each by a suitable power of $\hat{u}$
and a suitable power of $p$ to bring it into $\Z_p[\frac{\hat{u}}{p},\frac{\hat{v}}{p}]$.

\begin{dfn}\label{fk}
Define 
    \begin{align*}
    c_{k}&=\prod_{i=0}^{k-1}\frac{\hat{v}-\hat{q}^{i}\hat{u}}{\hat{q}^k-\hat{q}^{i}},\\
    f_{k}&=p^{\nu_p(k!)}c_{k}.
    \end{align*}
\end{dfn}

\noindent It is easy to check that the elements $f_{k}$ lie in $\Z_p[\frac{\hat{u}}{p},\frac{\hat{v}}{p}]$ for all $k\in\N_0$, using
that $\nu_p\left(\prod_{i=0}^{k-1}(\hat{q}^k-\hat{q}^{i})\right)=\nu_p(k!)+k$.

\begin{dfn}\label{basis elements}
Let
    $$
    F_{i,j,k}=\hat{u}^i\left(\frac{\hat{u}}{p}\right)^jf_{k}.
    $$
\end{dfn}   
    
\noindent Now the method of Adams leads to the analogue of~\cite[Part III, Proposition 17.6]{AdamsSH}.

\begin{thm}\label{basis}
\begin{list}
{(\arabic{itemcounter})}
{\usecounter{itemcounter}\leftmargin=0.5em}
\item The intersection of the subring satisfying condition (1) of Proposition~\ref{subring} with $\Q_p[\hat{u},\hat{v}]$ is free on the $\Z_p[\hat{u}]$-basis $\{c_{k}\,|\,k\geqslant 0\}$.
\item A $\Z_p$-basis for $\pi_*(\ell\wedge\ell)/\textup{Torsion}$ is given by the polynomials $F_{i,j,k}$,
where $k\geqslant 0$, $0\leqslant j\leqslant \nu_p(k!)$ with $i=0$ if $j<\nu_p(k!)$ and $i\geqslant0$ if $j=\nu_p(k!)$.
\end{list}
\end{thm}

\begin{proof}
Firstly, it follows from~\cite[Proposition 4.2]{CCW2} that the elements $c_{k}$ satisfy 
Proposition~\ref{subring} condition $(1)$, but will not do so if divided by
more $p$'s. It follows that $\left(\frac{\hat{u}}{p}\right)^{\nu_p(k!)}f_{k}$ satisfies
this condition, but $\left(\frac{\hat{u}}{p}\right)^{\nu_p(k!)+1}f_{k}$ does not.

To prove part (1), note that the $c_{k}$ are clearly linearly independent. Consider a polynomial $f(u,v)\in\Q_p[\hat{u},\hat{v}]$ 
satisfying condition (1) of Proposition~\ref{subring} and suppose that  $f$ is homogeneous of degree $\rho n$.
We can write $f$ as 
    $$
    f(\hat{u},\hat{v})=\lambda_0\hat{u}^n+\lambda_1\hat{u}^{n-1}c_{1}+\lambda_2\hat{u}^{n-2}c_{2}+\cdots.
    $$ 
Assume as an inductive hypothesis that $\lambda_0,\lambda_1,\ldots,\lambda_{s-1}$ lie in $\Z_p$. Let the sum of the remaining terms be         
    $$
    g(\hat{u},\hat{v})=\lambda_s\hat{u}^{n-s}c_{s}+\lambda_{s+1}\hat{u}^{n-s-1}c_{s+1}+\cdots.
    $$ 
This sum must also satisfy condition (1) of Proposition~\ref{subring}. Thus    
$g(t,\hat{q}^st)=\lambda_s t^n\in\Z_p[t]$ 
and hence $\lambda_s\in\Z_p$. The initial case for $\lambda_0$ works in the same way and this completes the induction.
Thus we can write  $f$ as a $\Z_p[\hat{u}]$-linear combination of the $c_{k}$s. 

To prove part (2), write
    $$
    n_k:=\text{ numerator of }c_{k}=\prod_{i=0}^{k-1}(\hat{v}-\hat{q}^{i}\hat{u}).
    $$ 
In degree $\rho k$ there are $k+1$ $\Q_p$-basis elements 
    $$
    n_k,\hat{u}n_{k-1},\hat{u}^2n_{k-2},\ldots,\hat{u}^k.
    $$ 
In order to produce the elements $F_{i,j,k}$ we divided each of the above elements by the highest power of $p$ which leaves it satisfying both conditions (1) and (2) of Proposition~\ref{subring}. For the element $\hat{u}^in_s$, this is $\min\{p^{s+\nu_p(s!)},p^{s+i}\}$. Now consider an element $f(u,v)\in\Q_p[\hat{u},\hat{v}]$, homogeneous of degree $\rho k$, which satisfies conditions (1) and (2) of Proposition~\ref{subring}. We can write $f$ as     
    $$
    f(\hat{u},\hat{v})=\frac{\lambda_0}{p^{a_0}}\hat{u}^k+\frac{\lambda_1}{p^{a_1}}\hat{u}^{k-1}n_1
        +\frac{\lambda_2}{p^{a_2}}\hat{u}^{k-2}n_2+\cdots
    $$ 
where $\lambda_i\in\Z_p$ for $i\geqslant 0$. By part (1), $a_s\leqslant \nu_p(\text{denominator of }c_{s})=s+\nu_p(s!)$.
We also claim that $a_s\leqslant (k-s)+s=k$.
Let the inductive hypothesis for a downwards induction be that $a_{s'}\leqslant k$ for $s'>s$. Let the sum of the remaining terms be     
    $$
    g(\hat{u},\hat{v})=\frac{\lambda_0}{p^{a_0}}\hat{u}^k+\cdots+\frac{\lambda_s}{p^{a_s}}\hat{u}^{k-s}n_s,
    $$ 
which must also satisfy conditions (1) and (2) of Proposition~\ref{subring}. The top coefficient $\frac{\lambda_s}{p^{a_s}}$ is the coefficient of $\hat{u}^{k-s}\hat{v}^s$ so because $g$ satisfies condition (2) of Proposition~\ref{subring} we must have that $a_s\leqslant (k-s)+s=k$. The first step of the induction works in the same way and the induction is complete.
Thus $f$ is a $\Z_p$-linear combination of the elements $F_{i,j,k}$.
\end{proof}

\subsection{Properties of the Basis}

We now consider how the basis we have found above relates to Kane's splitting of $\ell\wedge\ell$.

\begin{dfn}
Let 
    $$
    G_{m,n}=\frac{\pi_m(\ell\wedge\K(n))}{\text{Torsion}}.
    $$ 
Then we have $$G_{*,*}=\bigoplus_{m,n}G_{m,n}\cong\frac{\pi_*(\ell\wedge\ell)}{\text{Torsion}}.$$
\end{dfn}

\begin{prop}\label{homotopyGmn}
For each $n\geqslant0$,
\begin{displaymath}
G_{m,n}=
\begin{cases}
\Z_p & \textrm{ if } m \textrm{ is a multiple of $\rho$ and } m\geqslant\rho n,\\
0 & \textrm{ otherwise.}
\end{cases}
\end{displaymath}
\end{prop}

\begin{proof}
This follows directly from the description of $\frac{\pi_*(\ell\wedge K(n))}{\text{Torsion}}$ given in
~\cite[Proposition 9:2]{Kane}.
\end{proof}

\begin{dfn}\label{gml}
For $m\geqslant l$, define the element $g_{m,l}\in\Z_p[\frac{\hat{u}}{p},\frac{\hat{v}}{p}]$ to be the 
element produced from $f_{l}$ lying in degree $\rho m$, i.e.
\begin{displaymath}
g_{m,l}=\left\{\begin{array}{ll}
F_{0,m-l,l} & \textrm{ if }  m\leqslant \nu_p(l!)+l,\\
F_{m-l-\nu_p(l!),\nu_p(l!),l} & \textrm{ if }  m>\nu_p(l!)+l.
\end{array} \right.
\end{displaymath}
\end{dfn}

\begin{lem}\label{mhomotopy}
The elements $\{g_{m,l}:0\leqslant l\leqslant m\}$ form a basis for $G_{\rho m,*}$.
\end{lem}

\begin{proof}
The elements $\{g_{m,l}:0\leqslant l\leqslant m\}$ are precisely all of the basis elements $F_{i,j,k}$ which lie in homotopy degree $\rho m$.
\end{proof}

We will need to see that $\pi_*(\ell\wedge \ell)$ contains no torsion of order larger than $p$. To prove
this we need information about the stable class of $H^*(\ell)$.

\begin{prop}\label{stable l}
The stable isomorphism class of $H^*(\ell)$ as a $B$-module is
    $$
    \bigotimes_{i=0}^\infty\bigoplus_{j=0}^{p-1}\Sigma^{j(2p^{i+1}-2p^{i}-\pi_p(i))}I^{j\pi_p(i)},
    $$ 
where $\pi_p(i)=\frac{p^i-1}{p-1}$.
\end{prop}

\begin{proof}
We calculate the $Q_0$ and $Q_1$ homologies of $H_{-*}(\ell)$ explicitly and 
deduce its stable class and then dualise this statement to find the stable class of $H^*(\ell)$.
Recall that
    $$
    \pi_*(\ell\wedge H\Z/p)\cong H_*(\ell)\cong\Z/p[\bar{\xi_1},\bar{\xi_2},\ldots]\otimes E(\bar{\tau_2},\bar{\tau_3},\ldots),
    $$
where a bar over an element denotes the image of that element under the anti-automorphism $\chi$ of the
dual Steenrod algebra $\A_p^*$; see~\cite{Knapp}.
The action of $Q_0$ sends $\bar{\xi}_i$ to zero and $\bar{\tau}_i$ to $-\bar{\xi}_i$. Using the K\"{u}nneth Theorem, we see
that the $Q_0$ homology of $\pi_{-*}(\ell\wedge H\Z/p)$ is isomorphic to $\Z/p[\bar{\xi_1}]$.
The action of $Q_1$ sends $\bar{\xi}_i$ to zero and $\bar{\tau}_i$ to $-\bar{\xi}_{i-1}^p$. 
Again using the K\"{u}nneth Theorem, we see that the $Q_1$ homology of $\pi_{-*}(\ell\wedge H\Z/p)$ is isomorphic to ${\Z/p[\bar{\xi}_1,\bar{\xi}_2,\ldots]}/{(\bar{\xi}_1^p,\bar{\xi}_2^p,\ldots)}$.

Recall from~\cite[Part III, Proposition 16.4]{AdamsSH} that there are so-called lightning flash modules $M_i$
characterised as follows.
For $i\geq1$, $M_i$ is a finite-dimensional submodule of $\pi_{-*}(\ell\wedge H\Z/p)$ such that
\begin{list}
{(\arabic{itemcounter})}
{\usecounter{itemcounter}\leftmargin=0.5em}
\item $H(M_i;Q_0)\cong\Z/p$ generated by $\bar{\xi}_1^{p^{i-1}}$ and
\item $H(M_i;Q_1)\cong\Z/p$ generated by $\bar{\xi}_i$.
\end{list}
It follows that there is a stable isomorphism
    $$
    M_i\approxeq\Sigma^{-2p^i+2p^{i-1}+\pi_p(i-1)}I^{-\pi(i-1)}.
    $$
We claim that the map
    $$
    \bigotimes_{i=1}^\infty\bigoplus_{j=0}^{p-1}M_i^j\rightarrow
    \pi_{-*}(\ell\wedge H\Z/p)
    $$
is a stable isomorphism.
Indeed it is clear that it is a $B$-module map and, again using the K\"{u}nneth
Theorem, we see that it induces an isomorphism on both $Q_0$ and $Q_1$ homology,
so it is a stable isomorphism by~\cite[Part III, Lemma 16.7]{AdamsSH}.
 
Dualising, we get a stable isomorphism 
    $$
    H^*(\ell)\approxeq(1+M_1^*+{M_1^*}^2+\cdots+{M_1^*}^{p-1})(1+M_2^*+{M_2^*}^2+\cdots+{M_2^*}^{p-1})\ldots,
    $$ 
and combining this with  
    $$
    M_i^*\approxeq\Sigma^{2p^i-2p^{i-1}-\pi_p(i-1)}I^{\pi(i-1)}
    $$
gives the result.
\end{proof}

\begin{prop}\label{p torsion}
Let $\tilde{G}_{m,n}=\pi_m(\ell\wedge\K(n))$, then $\tilde{G}_{m,n}\cong G_{m,n}\oplus W_{m,n}$ where $W_{m,n}$ is a finite elementary abelian $p$-group, i.e. $\tilde{G}_{m,n}$ contains no torsion of order larger than $p$.
\end{prop}

\begin{proof}
This is proved for the case $p=2$ in \cite[Part III, Chapter 17]{AdamsSH}; the odd primary analogue is similar. We require two conditions in order to apply the two results of Adams necessary to prove this. Firstly that $H_r(\ell\wedge\ell;\Z)$ is finitely generated for each $r$ which is true (see \cite[p.353]{AdamsSH}) and secondly that, as a $B$-module, $H^*(\ell\wedge\ell)$ is stably isomorphic to $\oplus_i\Sigma^{a(i,p)}I^{b(i,p)}$ where $b(i,p)\geqslant0$ and $a(i,p)+b(i,p)\equiv0\mod2$.

For the second condition, using the K\"{u}nneth formula and Proposition~\ref{stable l},
    $$
    H^*(\ell\wedge\ell)\cong H^*(\ell)\otimes H^*(\ell)\approxeq
        \left(
        \bigotimes_{i=0}^\infty\bigoplus_{j=0}^{p-1}\Sigma^{j(2p^{i+1}-2p^{i}-\pi_p(i))}I^{j\pi_p(i)}
        \right)^{\otimes 2}.
    $$
Then it is clear that 
the required condition holds for each summand and hence for $H^*(\ell\wedge\ell)$.

Given these assumptions we can now apply \cite[Part III, Lemma 17.1]{AdamsSH} which states that $H_*(ku\wedge\ell;\Z)$ and hence $H_*(\ell\wedge\ell;\Z)$ has no torsion of order higher than $p$. Then from \cite[Part III, Proposition 17.2(i)]{AdamsSH}, the Hurewicz homomorphism $$h:\pi_*(ku\wedge\ell)\rightarrow H_*(ku\wedge\ell;\Z)$$ is a monomorphism. Since this is true of $ku\wedge\ell$ it follows that the same is true of $\ell\wedge\ell$ and so the result follows.
\end{proof}

\begin{dfn}
Consider the projection map 
    $$
    \ell\wedge\ell\simeq\bigvee_{n\geqslant0}\ell\wedge\K(n)\rightarrow\ell\wedge\K(n)
    $$ 
and let $P_n:G_{*,*}\rightarrow G_{*,n}$ be the induced projection map on homotopy modulo torsion.
\end{dfn}

\begin{lem}\label{proj}
$P_n(g_{n,l})=0$ if $l<n$.
\end{lem}

\begin{proof}
Since $G_{m,n}$ is torsion free we can consider just whether $P_n(g_{n,l})$ is zero in $G_{*,n}\otimes\Q_p$. Let $l<n$, then for $\alpha(n,l)\in\N_0$,
    $$
    g_{n,l}=\frac{\hat{u}^{n-l}}{p^{\alpha(n,l)}}f_{l}\in 
    \hat{u}^{n-l}\frac{\pi_{\rho l}(\ell\wedge\ell)}{\text{Torsion}}\otimes\Q_p
    \subset\frac{\pi_{\rho n}(\ell\wedge\ell)}{\text{Torsion}}\otimes\Q_p.
    $$
Since $P_n$ is a left $\ell$-module map, $P_n(g_{n,l})$ is $\hat{u}^{n-l}$ times an element of 
$G_{\rho l, n}$. But this group is zero, by Proposition~\ref{homotopyGmn}.
\end{proof}

\subsection{The Elements $z_{m}$}

Now we choose labels for generators of certain homotopy groups and compare with our basis elements.

\begin{dfn}
Let $z_{n}$ be a generator for $G_{\rho n,n}\cong\Z_p$ and let $\tilde{z}_{n}$ be any element in $\tilde{G}_{\rho n,n}\cong G_{\rho n,n}\oplus W_{\rho n,n}$ where the first co-ordinate is $z_{n}$.
\end{dfn}

\begin{prop}\label{0 or 1}
In the Adams spectral sequence     
    $$
    E_2^{s,t}\cong\Ext_{\A_p}^{s,t}(H^*(\ell\wedge\K(n)),\Z/p)\Longrightarrow\pi_{t-s}(\ell\wedge\K(n))\otimes\Z_p
    $$
the class of $\tilde{z}_{n}$ is represented in either $E_2^{0,\rho n}$ or $E_2^{1,\rho n+1}$.
\end{prop}

\begin{proof}
Consider the Adams spectral sequence
    \begin{align*}
        E_2^{s,t}&=\Ext_{\A_p}^{s,t}(H^*(\ell\wedge\K(n)),\Z/p)\\
               & \cong\Ext_{\A_p}^{s,t}(H^*(\ell)\otimes H^*(\K(n)),\Z/p)\\
               &\cong\Ext_B^{s,t}(H^*(\K(n)),\Z/p) \qquad \qquad\Longrightarrow\pi_{t-s}(\ell\wedge\K(n))\otimes\Z_p.
    \end{align*}
By Theorem~\ref{stable Kn}, for $s>0$,
    $$
        E_2^{s,t}\cong\Ext_B^{s,t}(\Sigma^{\rho n-\nu_p(n!)}I^{\nu_p(n!)},\Z/p)
        \cong\Ext_B^{s+\nu_p(n!),t-\rho n+\nu_p(n!)}(\Z/p,\Z/p).
   $$
We see that, away from $s=0$, the $E_2$ term is isomorphic to a shifted version of $\Ext_B^{*,*}(\Z/p,\Z/p)\cong\Z/p[c,d]$ with $c\in\Ext_B^{1,1}$ and $d\in\Ext_B^{1,2p-1}$. The spectral sequence collapses at $E_2$, just as for~(\ref{ASS}).
The spectral sequence gives us information about a filtration $F^i$ of $\tilde{G}_{\rho n,n}=\pi_{\rho n}(\ell\wedge\K(n))\otimes\Z_p$.
The multiplicative structure of the spectral sequence is such that $pF^i= F^{i+1}$,
for $i\geqslant 1$ and $F^1\cong \Z_p$. By Proposition~\ref{p torsion}, $W_{\rho n,n}$ is an elementary abelian $p$-group, so $pW_{\rho n,n}=0$ and $W_{\rho n,n}$ must be represented in $E_2^{0,\rho n}$.

Assume that the generator $\tilde{z}_{n}$ is represented in $E_2^{j,\rho n+j}$ for $j\geqslant2$, then we must have that $\tilde{z}_{n}\in F^j$. Then there is some generator $\tilde{z}_{n}'\in F^1$ such that $p^{j}\tilde{z}_{n}'$ is a generator for $F^{j+1}$. Therefore, for some $\gamma\in\Z_p$, 
    $$
    p^j\gamma\tilde{z}_{n}'=p\tilde{z}_{n}.
    $$ 
Thus
    $$
    p(p^{j-1}\gamma\tilde{z}_{n}'-\tilde{z}_{n})=0
    $$ 
and hence we must have $p^{j-1}\gamma\tilde{z}_{n}'-\tilde{z}_{n}\in W_{\rho n,n}$ because nothing else has any torsion. This implies that in $G_{\rho n,n}$, $z_{n}$ has a factor $p$ which contradicts the fact that we chose $z_{n}$ to be a generator of $G_{\rho n,n}\cong\Z_p$.
\end{proof}

We can now give a more explicit description of the generators $z_{n}$ in terms of our basis elements $g_{m,l}$. We first state
a lemma we will need; this is readily proved from the definitions.

\begin{lem}\label{alglem}
For $0\leqslant i\leqslant m-1$,
\footnotesize
\begin{displaymath}
\left(\frac{\hat{u}}{p}\right)^{\nu_p(m!)}g_{m,i}=\begin{cases}
\frac{1}{p^{\nu_p(m!)+m-\nu_p(i!)-i}}\hat{u}^{\nu_p(m!)-\nu_p(i!)+m-i}\left(\frac{\hat{u}}{p}\right)^{\nu_p(i!)}f_{i} & \textrm{ if } m\leqslant \nu_p(i!)+i,\\
\frac{1}{p^{\nu_p(m!)}}\hat{u}^{\nu_p(m!)-\nu_p(i!)+m-i}\left(\frac{\hat{u}}{p}\right)^{\nu_p(i!)}f_{i} & \textrm{ if }m>\nu_p(i!)+i.
\end{cases}
\end{displaymath}\qed
\end{lem}
\normalsize

\begin{dfn}
For $m,i\geq 0$, define $\beta(m,i)\in\N_0$ by
    \begin{displaymath}
        \beta(m,i)=\left\{\begin{array}{ll}
        \nu_p(m!) & \textrm{ if } m>\nu_p(i!)+i,\\
        \nu_p(m!)+m-\nu_p(i!)-i & \textrm{ if } m\leqslant \nu_p(i!)+i.
        \end{array}\right.
    \end{displaymath}

\end{dfn}

\begin{prop} \label{z in basis elements}
The generators $z_{m}\in G_{\rho m, *}$ have the following form 
    $$
    z_{m}=\sum_{i=0}^m p^{\beta(m,i)}\lambda_{m,i}g_{m,i}
    $$
where $\lambda_{s,t}\in\Z_p$ if $s\neq t$, $\lambda_{s,s}\in\Z_p^\times$.
\end{prop}

\begin{proof}
By Lemma~\ref{mhomotopy}, $\{g_{m,l}:0\leqslant l\leqslant m\}$ forms a basis for $G_{\rho m,*}$, so 
we can express $z_{m}$ in terms of this basis, 
    \begin{equation}\label{z in basis}
    z_{m}=\lambda_{m,m}g_{m,m}+
            \tilde{\lambda}_{m,m-1}g_{m,m-1}+\cdots
            +\tilde{\lambda}_{m,0}g_{m,0},
    \end{equation}
where $\lambda_{m,m},\tilde{\lambda}_{m,l}\in\Z_p$. The projection map $P_m:G_{\rho m,*}\rightarrow G_{\rho m,m}$ 
acts as the identity on $z_{m}$ and as zero on all $g_{m, l}$ where $m\neq l$ by Lemma~\ref{proj}. Hence
    $$
    z_{m}=P_m(z_{m})=\lambda_{m,m}P_m(g_{m,m})=\lambda_{m,m}P_m(f_{m}).
    $$
Thus the coefficient $\lambda_{m,m}$ is a unit, since otherwise $z_{m}$ would have a factor of $p$, contradicting its choice as a generator of $G_{\rho m,m}\cong\Z_p$. 

 We can now multiply by the largest power of $\frac{\hat{u}}{p}$ possible to leave the result still lying in $G_{*,m}$ and we get     
    $$
    \left(\frac{\hat{u}}{p}\right)^{\nu_p(m!)}z_{m}
    =\lambda_{m,m}P_m\left(\left(\frac{\hat{u}}{p}\right)^{\nu_p(m!)}f_{m}\right)
    $$ 
 which lies in $G_{\rho(\nu_p(m!)+m),m}$. By multiplying equation~\eqref{z in basis} by $\left(\frac{\hat{u}}{p}\right)^{\nu_p(m!)}$ we now have the following relation in $G_{\rho(\nu_p(m!)+m),m}\otimes\Q_p$
    \begin{equation}\label{coefficients}
    \left(\frac{\hat{u}}{p}\right)^{\nu_p(m!)}z_{m}=
    \left(\frac{\hat{u}}{p}\right)^{\nu_p(m!)}\lambda_{m,m}f_{m}+
    \sum_{i=0}^{m-1}\tilde{\lambda}_{m,i}\left(\frac{\hat{u}}{p}\right)^{\nu_p(m!)}g_{m,i}.
    \end{equation}

Since the left hand side of this equation lies in $G_{\rho(\nu_p(m!)+m),m}$, we can calculate how many factors of $p$ each 
$\tilde{\lambda}_{m,i}$ must have to ensure that, once the right hand side is expressed in terms of the basis in Theorem~\ref{basis}, all the coefficients are $p$-adic integers.

Using Lemma~\ref{alglem} we can see that if $m\leqslant \nu_p(i!)+i$ we have 
    $$
    \left(\frac{\hat{u}}{p}\right)^{\nu_p(m!)}g_{m,i}=\frac{1}{p^{\nu_p(m!)+m-\nu_p(i!)-i}}\hat{u}^{\nu_p(m!)-\nu_p(i!)+m-i}
    \left(\frac{\hat{u}}{p}\right)^{\nu_p(i!)}f_{i}
    $$ 
and so in equation~\eqref{coefficients} the coefficient $\tilde{\lambda}_{m,i}$ must be divisible by $p^{\nu_p(m!)+m-\nu_p(i!)-i}$.

Similarly, when $m>\nu_p(i!)+i$ we have 
    $$
    \left(\frac{\hat{u}}{p}\right)^{\nu_p(m!)}g_{m,i}=\frac{1}{p^{\nu_p(m!)}}\hat{u}^{\nu_p(m!)-\nu_p(i!)+m-i}
    \left(\frac{\hat{u}}{p}\right)^{\nu_p(i!)}f_{i}
    $$ 
and so in equation~\eqref{coefficients}, $\tilde{\lambda}_{m,i}$ must be divisible by $p^{\nu_p(m!)}$.
\end{proof}

\begin{prop}\label{0}
In Proposition~\ref{0 or 1}, $\tilde{z}_{n}$ is actually represented in $E_2^{0,\rho n}$.
\end{prop}

\begin{proof}
We will assume that $\tilde{z}_{n}$ is represented in $E_2^{1,\rho n+1}$ in the spectral sequence and obtain a contradiction. 
The spectral sequence in question collapses and, away from $s=0$, the $E_2$ page is  
    $$
    E_2^{s,t}\cong\Ext_B^{s+\nu_p(n!),t-\rho n+\nu_p(n!)}(\Z/p,\Z/p).
    $$ 
This is a shifted version of $\Ext_B^{*,*}(\Z/p,\Z/p)\cong\Z/p[c,d]$, so on the line $s=1$ the non-zero groups are $E_2^{1,\rho n+1}$, $E_2^{1,\rho(n+1)+1}$, $\ldots$, $E_2^{1,\rho(n+\nu_p(n!)+1)+1}$ and each of these is a single copy of $\Z/p$. Using the multiplicative structure of the spectral sequence, if there is a class $w\in E_2^{j,\rho(n+j)+1}$ and the group $E_2^{j,\rho(n+j+1)+1}$ is non-zero then there exists a class $w'\in E_2^{j,\rho(n+j+1)+1}$ such that $pw'=\hat{u}w$. In other words if $w$ is represented by $c^xd^y$ in the spectral sequence where $x\geqslant1$, $y\geqslant0$ then $w'$ is represented by $c^{x-1}d^{y+1}$. Applying  this to $\tilde{z}_{n}\in E_2^{1,\rho n+1}$, since  $E_2^{1,\rho(n+\nu_p(n!)+1)+1}$ is non-zero, there must exist a class $w\in E_2^{1,\rho(n+\nu_p(n!)+1)+1}$ such that     
    $$
    \hat{u}^{1+\nu_p(n!)}\tilde{z}_{n}=p^{1+\nu_p(n!)}w.
    $$ 
This implies that $\hat{u}^{1+\nu_p(n!)}\tilde{z}_{n}$ is divisible by $p^{1+\nu_p(n!)}$ in $G_{*,*}$. However this contradicts the proof of Proposition~\ref{z in basis elements}, hence $\tilde{z}_{n}$ must be represented in $E_2^{0,\rho n}$.
\end{proof}

\begin{lem}\label{zonly}
In the spectral sequence
    $$
    E_2^{s,t}\cong\Ext^{s,t}_B(H^*(\K(n)),\Z/p)\Longrightarrow\pi_{t-s}(\ell\wedge\K(n))\otimes\Z_p,
    $$
up to multiplication by a unit, $\left(\frac{\hat{u}}{p}\right)^i(p\tilde{z}_{n})$ is represented by $c^{\nu_p(n!)+1-i}d^i$ for $0\leqslant i\leqslant\nu_p(n!)$ and $\hat{u}^j\left(\frac{\hat{u}}{p}\right)^{\nu_p(n!)}(\tilde{z}_{n})$ is represented by $d^{\nu_p(n!)+j}$ for $j\geqslant1$.
\end{lem}

\begin{proof}
From Proposition~\ref{0} we know that in the spectral sequence     
    $$
    E_2^{s,t}\cong\Ext^{s,t}_B(H^*(\K(n);\Z/p),\Z/p)\Longrightarrow\pi_{t-s}(\ell\wedge\K(n))\otimes\Z_p
    $$ 
$\tilde{z}_{n}$ is represented in $E_2^{0,\rho n}$. By the multiplicative structure of the spectral sequence this means that $p\tilde{z}_{n}$ is represented in $E_2^{1,\rho n+1}$. We have
    $$
    E_2^{1,\rho n+1}\cong\Ext_B^{1+\nu_p(n!),1+\nu_p(n!)}(\Z/p,\Z/p)
        \cong\Z/p\langle c^{1+\nu_p(n!)}\rangle.
    $$
We know from~\cite[Part III, Lemma 17.11]{AdamsSH} that in the spectral sequence, multiplication by $c$ and $d$ correspond to multiplication by $p$ and $\hat{u}$ respectively on homotopy groups. We list below some homotopy elements of $\pi_*(\ell\wedge\K(n))$ with a choice of corresponding representatives in the spectral sequence.
\begin{displaymath}
\begin{tabular}{c|c}
Homotopy element & Representative\\
\hline
&\\
$p\tilde{z}_{n}$ & $c^{1+\nu_p(n!)}$ \\
$\left(\frac{\hat{u}}{p}\right)(p\tilde{z}_{n})$ & $c^{\nu_p(n!)}d$ \\
$\left(\frac{\hat{u}}{p}\right)^2(p\tilde{z}_{n})$ & $c^{\nu_p(n!)-1}d^2$ \\
\vdots & \vdots \\
$\left(\frac{\hat{u}}{p}\right)^{\nu_p(n!)}(p\tilde{z}_{n})$ & $cd^{\nu_p(n!)}$ \\
$\hat{u}\left(\frac{\hat{u}}{p}\right)^{\nu_p(n!)}(\tilde{z}_{n})$ & $d^{\nu_p(n!)+1}$ \\
$\hat{u}^2\left(\frac{\hat{u}}{p}\right)^{\nu_p(n!)}(\tilde{z}_{n})$ & $d^{\nu_p(n!)+2}$ \\
\vdots & \vdots \\
\end{tabular}
\end{displaymath}
From this table it is clear to see that the descriptions given in the statement are correct.
\end{proof}

Recall from Definition~\ref{iota} the maps $$\iota_{m,n}:\ell\wedge\K(m)\rightarrow \ell\wedge\K(n)$$ which were maps represented in the spectral sequence
    \begin{align*}
    E_2^{s,t}\cong\Ext_B^{s,t}(H^*(D(\K(m)))&\otimes H^*(\K(n)),\Z/p)\\
    &\Longrightarrow\pi_{t-s}(D(\K(m))\wedge\K(n)\wedge \ell)\otimes\Z_p
    \end{align*}
by a choice of generator of 
    $$
    E_2^{m-n-\nu_p(n!)+\nu_p(m!),m-n-\nu_p(n!)+\nu_p(m!)}.
    $$ 

\begin{prop}\label{iota on z}
For $m>n$, the map induced in the $(\rho m)$th homotopy group $$(\iota_{m,n})_*:G_{\rho m,m}\rightarrow G_{\rho m, n}$$ satisfies the condition 
    $$
    (\iota_{m,n})_*(z_{m})=\mu_{m,n}p^{\nu_p(m!)-\nu_p(n!)}\hat{u}^{m-n}z_{n}
    $$ 
for some $p$-adic unit $\mu_{m,n}$.
\end{prop}

\begin{proof}
By definition, $\tilde{z}_{m}$ is any element in $G_{\rho m,m}\oplus W_{\rho m,m}$ whose first co-ordinate is $z_{m}$. Also $W_{\rho m,m}$ has torsion of order $p$ at the highest by Proposition~\ref{p torsion}. We will prove the analogous result for the element $p\tilde{z}_{m}=pz_{m}$; then by linearity the required result will be true for $z_{m}$.

By Lemma~\ref{zonly}, in the spectral sequence     
    $$
    E_2^{s,t}\cong\Ext^{s,t}_B(H^*(\K(m)),\Z/p)\Longrightarrow\pi_{t-s}(\ell\wedge\K(m))\otimes\Z_p,
    $$ 
$p\tilde{z}_{m}$ is represented in $E_2^{1,\rho m+1}\cong\Ext_B^{1+\nu_p(m!),1+\nu_p(m!)}(\Z/p,\Z/p)$, up to a unit, by $c^{1+\nu_p(m!)}$.

Recall that in the spectral sequence
    \begin{align*}
    E_2^{s,t}\cong\Ext_B^{s,t}(H^*(D(\K(m));\Z/p)&\otimes H^*(\K(n);\Z/p),\Z/p)\\
            &\Longrightarrow\pi_{t-s}(D(\K(m))\wedge\K(n)\wedge \ell)\otimes\Z_p,
    \end{align*}
the maps $\iota_{m,n}:\ell\wedge\K(m)\rightarrow \ell\wedge\K(n)$ are represented in
\begin{align*}
E_2^{m-n-\nu_p(n!)+\nu_p(m!),m-n-\nu_p(n!)+\nu_p(m!)}&\cong\Ext_B^{m-n,(m-n)(\rho+1)}(\Z/p,\Z/p)\\
&\cong\Z/p\langle d^{m-n}\rangle.
\end{align*}

Using the pairing of Ext groups described in the proof of Proposition~\ref{extprod},
    $$
    \Ext^{s,t}(\Sigma^aI^b,\Z/p)\otimes \Ext^{s',t'}(\Sigma^{a'}I^{b'},\Z/p)\rightarrow\Ext^{s+s',t+t'}(\Sigma^{a+a'}I^{b+b'},\Z/p)
    $$ 
we get an induced pairing on the $E_2$ pages of the respective Adams spectral sequences. Since in all cases the spectral sequences collapse this passes to the $E_\infty$ pages. The pairing also respects filtrations, so the Ext group pairing passes to a pairing of spectral sequences, giving us a map
    \small\begin{align*}
    \Ext_B^{s,t}(H^*(D(\K(m)))\otimes H^*(\K(n)),\Z/p)&\otimes\Ext^{s',t'}_B(H^*(\K(m)),\Z/p)\\
    &\rightarrow\Ext_B^{s+s',t+t'}(H^*(\K(n)),\Z/p).
    \end{align*}\normalsize
This shows that $(\iota_{m,n})_*(p\tilde{z}_{m})$ is represented in the spectral sequence     
    $$
    E_2^{s,t}\cong\Ext_B^{s,t}(H^*(\K(n)),\Z/p)\Longrightarrow\pi_{t-s}(\ell\wedge\K(n))\otimes\Z_p
    $$ 
by a generator of
    \begin{align*}
    &E_2^{1+m-n-\nu_p(n!)+\nu_p(m!),\rho m+1+m-n-\nu_p(n!)+\nu_p(m!)}\\
    &\quad\cong\Ext_B^{1+m-n-\nu_p(n!)+\nu_p(m!),\rho m+1+m-n-\nu_p(n!)+\nu_p(m!)}(H^*(\K(n)),\Z/p)\\
    &\quad\cong\Ext_B^{1+m-n-\nu_p(n!)+\nu_p(m!),\rho m+1+m-n-\nu_p(n!)+\nu_p(m!)}(\Sigma^{\rho n-\nu_p(n!)}I^{\nu_p(n!)},\Z/p)\\
    &\quad\cong\Ext_B^{1+m-n+\nu_p(m!),\rho m+1+m-n+\nu_p(m!)-\rho n}(\Z/p,\Z/p)\\
    &\quad\cong\Ext_B^{1+m-n+\nu_p(m!),1+(\rho+1)(m-n)+\nu_p(m!)}(\Z/p,\Z/p)\\
    &\quad\cong\Z/p\langle c^{1+\nu_p(m!)}d^{m-n}\rangle.
    \end{align*}
Thus $(\iota_{m,n})_*(p\tilde{z}_{m})$ is, up to a unit, represented by $c^{1+\nu_p(m!)}d^{m-n}$ and all that remains is to express this element in terms of $p\tilde{z}_{n}$.

Using Lemma~\ref{zonly} we can see that we have two cases for $(\iota_{m,n})_*(p\tilde{z}_{m})$, either the power of $d$ in its representative is at least $\nu_p(n!)+1$ (and hence the power of $c$ in its representative is zero) or not.

In the first case we have $m-n\geqslant \nu_p(n!)+1$. Then by Lemma \ref{zonly}, $d^{m-n}$ represents     
    $$
    \left(\frac{\hat{u}}{p}\right)^{\nu_p(n!)}\hat{u}^{m-n-\nu_p(n!)}\tilde{z}_{n}=p^{-\nu_p(n!)}\hat{u}^{m-n}\tilde{z}_{n}.
    $$ 
This implies that up to a $p$-adic unit, $(\iota_{m,n})_*(p\tilde{z}_{m})$ is equal to     
    $$
    p^{1+\nu_p(m!)}p^{-\nu_p(n!)}\hat{u}^{m-n}\tilde{z}_{n}=p^{\nu_p(m!)-\nu_p(n!)}\hat{u}^{m-n}(p\tilde{z}_{n}).
    $$

In the second case we have $m-n<\nu_p(n!)+1$. Hence by Lemma~\ref{zonly} the representative is $c^{1+\nu_p(n!)-m+n}d^{m-n}$ and this represents the homotopy element 
    $$
    \left(\frac{\hat{u}}{p}\right)^{m-n}(p\tilde{z}_{n}).
    $$ 
This gives us that up to a $p$-adic unit, $(\iota_{m,n})_*(p\tilde{z}_{m})$ is equal to
    $$
    p^{1+\nu_p(m!)-(1+\nu_p(n!)-m+n)}\left(\frac{\hat{u}}{p}\right)^{m-n}(p\tilde{z}_{n})
    =p^{\nu_p(m!)-\nu_p(n!)}\hat{u}^{m-n}(p\tilde{z}_{n}).\qedhere
    $$
\end{proof}

\section{The matrix of $1\wedge \Psi^q$}
\label{Secmatrix}

In this section we study the matrix corresponding to the map $1\wedge\Psi^q:\ell\wedge\ell\rightarrow\ell\wedge\ell$ under the isomorphism $\Lambda$ of Theorem~\ref{iso}. Firstly information on the form of this matrix
is obtained by comparing the effect of the maps $(\iota_{m,n})_*$
on the basis elements $z_{m}$ with the effect of the induced map $(1\wedge\Psi^q)_*$. 
Then it is shown that, by altering $\Lambda$ by a conjugation, the matrix can be given a particularly nice and simple
form.

A remark about our notation is in order. We are following~\cite{CCW2} in denoting by $\Psi^q$ the $\ell$ Adams operation which acts on 
$\pi_{2(p-1)k}(\ell)$ as multiplication by $q^{(p-1)k}=\hat{q}^k$.  Some authors write $\Psi^{\hat{q}}$ for this operation.
Our choice in~\cite{CCW2} was motivated by wanting to compare directly the $ku$ operations with the $\ell$ ones.
But of course the $l$ operation only depends on $\hat{q}$. 

\subsection{The effect of $1\wedge\Psi^q$ on the basis}

\begin{lem}\label{action on f}
For $m\geqslant1$,
    $$
    (1\wedge\Psi^q)_*(f_{m})=\hat{q}^mf_{m}+p^{\nu_p(m)}\hat{u}f_{m-1}.
    $$ 
\end{lem}

\begin{proof}
Using that the map $(1\wedge\Psi^q)_*$ fixes $\hat{u}$, multiplies $\hat{v}$ by $\hat{q}$ and is additive and multiplicative,
a straightforward calculation gives
    $$
    (1\wedge\Psi^q)_*(c_{m})=\hat{q}^mc_{m}+\hat{u}c_{m-1}.
    $$
Then
     \begin{align*}
        (1\wedge\Psi^q)_*(f_{m})&=\hat{q}^mf_{m}+p^{\nu_p(m!)-\nu_p((m-1)!)}\hat{u}f_{m-1}\\
        &=\hat{q}^mf_{m}+p^{\nu_p(m)}\hat{u}f_{m-1}.\qedhere
\end{align*}
\end{proof}

\begin{prop}\label{action on g}
The action of $(1\wedge\Psi^q)_*$ on the basis elements is as follows.
    \begin{displaymath}
            (1\wedge\Psi^q)_*(g_{m,m})=\left\{\begin{array}{ll}
                \hat{q}^mg_{m,m}+p^{\nu_p(m)+1}g_{m,m-1}& \textrm{ if } m>p,\\
                \hat{q}^mg_{m,m}+pg_{m,m-1} & \textrm{ if } m=p,\\
                \hat{q}^mg_{m,m}+g_{m,m-1} & \textrm{ if } 1\leqslant m\leqslant p-1,\\
                g_{0,0} & \textrm{ if } m=0.
                                    \end{array}\right.
    \end{displaymath}
And for $m>n$,
\small\begin{multline*}
(1\wedge\Psi^q)_*(g_{m,n})\\
=\begin{cases}
\hat{q}^ng_{m,n}+g_{m,n-1}& \text{ if } m>\nu_p(n!)+n,\\
\hat{q}^ng_{m,n}+p^{\nu_p(n!)+n-m}g_{m,n-1}& \text{ if } \nu_p((n-1)!)+n-1<m\leqslant \nu_p(n!)+n,\\
\hat{q}^ng_{m,n}+p^{\nu_p(n)+1}g_{m,n-1}& \text{ if } m\leqslant \nu_p((n-1)!)+n-1.
\end{cases}
\end{multline*}
\normalsize
\end{prop}

\begin{proof}
This is a matter of straightforward case-by-case calculation, using Definition~\ref{gml} and Lemma~\ref{action on f}.
\end{proof}

\subsection{The Coefficients of the Matrix}

Let $A\in U_\infty\Z_p$ be the matrix such that $\Lambda(A)=1\wedge\Psi^q$. The main result to be proved in this section
provides some restrictions on the form of the matrix $A$, by comparing actions on the basis elements $z_{m}$ from Proposition~\ref{z in basis elements}. 

The following lemma is needed in the proof.
It follows easily from the definitions. 

\begin{lem}\label{lower g}
\begin{equation*}
\hat{u}^{m-n}g_{n,i}=\begin{cases}
p^{m-n}g_{m,i} & \textrm{ if } n\leqslant m\leqslant \nu_p(i!)+i,\\
p^{\nu_p(i!)-n+i}g_{m,i} & \textrm{ if } n\leqslant \nu_p(i!)+i<m,\\
g_{m,i} & \textrm{ if } \nu_p(i!)+i<n\leqslant m. 
\end{cases}
\end{equation*}\qed 
\end{lem}

\begin{prop}\label{Aform}
The matrix $A$ corresponding to the map $1\wedge\Psi^q:\ell\wedge\ell\rightarrow\ell\wedge\ell$ under the isomorphism $\Lambda$ has the form
    \begin{displaymath}
        A=\left(\begin{array}{cccccc}
            1 & \upsilon_0 & a_{0,2} & a_{0,3} & a_{0,4} & \cdots\\
            0 & \hat{q}     & \upsilon_1   & a_{1,3} & a_{1,4} & \cdots\\
            0 & 0     & \hat{q}^2     & \upsilon_2   & a_{2,4} & \cdots\\
            0 & 0     & 0       & \hat{q}^3     & \upsilon_3   & \cdots\\
                \vdots & \vdots & \vdots & \vdots & \vdots & \ddots
                \end{array}\right),
        \end{displaymath}
where $\upsilon_i\in\Z_p^\times$ for all $i\geqslant 0$ and $a_{i,j}\in\Z_p$ for all $i,j\geqslant 0$.
\end{prop}

\begin{proof}
By definition of $\Lambda$,
    $$
        (1\wedge\Psi^q)_*(z_{m})=\sum_{n\leqslant m}A_{n,m}(\iota_{m,n})_*(z_{m}).
    $$
Using Propositions~\ref{z in basis elements} and~\ref{iota on z}, this becomes
    \begin{align}\label{equate}
        \sum_{i=0}^mp^{\beta(m,i)}\lambda_{m,i}(1\wedge\Psi^q)_*(g_{m,i})&=A_{m,m}\sum_{i=0}^mp^{\beta(m,i)}\lambda_{m,i}g_{m,i}\nonumber\\
        &+\sum_{n<m}\sum_{i=0}^nA_{n,m}\mu_{m, n}p^{\nu_p(m!)-\nu_p(n!)+\beta(n,i)}\hat{u}^{m-n}\lambda_{n,i}g_{n, i},
    \end{align}
where $\lambda_{m,i}\in\Z_p, \lambda_{m,m}\in\Z_p^\times$ and $\mu_{m,n}\in\Z_p^\times$.

We will determine information about the $A_{n,m}$s by equating coefficients in equation~\eqref{equate} and using
Proposition~\ref{action on g} and Lemma~\ref{lower g}. Firstly if $m=0$, then 
    $$
    \lambda_{0,0}=\lambda_{0,0}(1\wedge\Psi^q)_*(g_{0,0})=A_{0,0}\lambda_{0,0}g_{0,0}=A_{0,0}\lambda_{0,0},
    $$
so $A_{0,0}=1$.

We will now split the rest of the proof into three cases.
\smallskip

Case (i): $1\leqslant m\leqslant p-1$. Equating coefficients of $g_{m,m}$ gives
    $\hat{q}^m\lambda_{m,m}=A_{m,m}\lambda_{m,m}$, 
so $A_{m,m}=\hat{q}^m$.

\noindent Next, equating coefficients of $g_{m,m-1}$, and using
$\hat{u}g_{m-1,m-1}=g_{m,m-1}$ and $\beta(m,m)=\beta(m,m-1)=0$,
 gives
    $$
    \lambda_{m,m}+\lambda_{m,m-1}\hat{q}^{m-1}
        =\hat{q}^{m}\lambda_{m,m-1}+A_{m-1,m}\mu_{m,m-1}\lambda_{m-1,m-1}.
    $$
So
    $$
    A_{m-1,m}=\mu_{m,m-1}^{-1}\lambda_{m-1,m-1}^{-1}((\hat{q}^{m-1}-\hat{q}^{m})
        \lambda_{m,m-1}+\lambda_{m,m})\ \in \Z_p^\times.
    $$
\smallskip

Case (ii): $m=p$. Equating coefficients of $g_{m,m}$ gives
     $\hat{q}^m\lambda_{m,m}=A_{m,m}\lambda_{m,m}$, so we have $A_{m,m}=\hat{q}^m$ as before.
Then, equating coefficients of $g_{m,m-1}$ gives
    $$
    \lambda_{m,m}p+p\lambda_{m,m-1}\hat{q}^{m-1}
    =\hat{q}^{m}p\lambda_{m,m-1}+A_{m-1,m}\mu_{m,m-1}p\lambda_{m-1,m-1},
    $$
where this time we have used $\hat{u}g_{m-1,m-1}=g_{m,m-1}$, 
$\beta(m,m-1)=\nu_p(p!)=1$ and $\beta(m,m)=0$.
So 
    $$
    A_{m-1,m}=\mu_{m,m-1}^{-1}\lambda_{m-1,m-1}^{-1}((\hat{q}^{m-1}-\hat{q}^{m})
    \lambda_{m,m-1}+\lambda_{m,m})\ \in \Z_p^\times.
    $$ 
\smallskip

Case (iii): $m>p$. We find that $A_{m,m}=\hat{q}^m$ as before.
Then, equating coefficients of $g_{m,m-1}$ we find
    \begin{align*}
    \lambda_{m,m}p^{\nu_p(m)+1}+&p^{\nu_p(m)+1}\lambda_{m,m-1}\hat{q}^{m-1}\\
    &=\hat{q}^{m}p^{\nu_p(m)+1}\lambda_{m,m-1}+A_{m-1,m}\mu_{m,m-1}p^{\nu_p(m)+1}\lambda_{m-1,m-1},
    \end{align*}
where we have used $g_{m,m-1}=\frac{\hat{u}}{p}g_{m-1,m-1}$, and  
    $$
    \beta(m,m-1)=\nu_p(m!)+m-\nu_p((m-1)!)-(m-1)=\nu_p(m)+1.
    $$
So 
    $$
    A_{m-1,m}=\mu_{m,m-1}^{-1}\lambda_{m-1,m-1}^{-1}((\hat{q}^{m-1}-\hat{q}^{m})
    \lambda_{m,m-1}-\lambda_{m,m})\ \in \Z_p^\times.\qedhere
    $$ 
\end{proof}

\subsection{Conjugation}

We now complete the proof of the odd primary analogue of~\cite[Theorem $4.2$]{BaSn}, by conjugating to obtain a particularly
nice form for the matrix. The argument we give for this follows an idea suggested by Francis Clarke, see~\cite[Theorem $5.4.3$]{Vic'sBook}.

Firstly, let $E$ be the invertible diagonal matrix with
    $$
    E_{i,j}=\begin{cases}
            1&\text{if $i=j=0$,}\\
            v_0v_1\dots v_{i-1},&\text{if $i=j>0$},\\
            0,&\text{otherwise}. 
    \end{cases}
    $$
 
Then
\begin{displaymath}
EAE^{-1}=C=\left(\begin{array}{cccccc}
1 & 1 & c_{0,2} & c_{0,3} & c_{0,4} & \cdots\\
0 & \hat{q} & 1       & c_{1,3} & c_{1,4} & \cdots\\
0 & 0 & \hat{q}^2     & 1       & c_{2,4} & \cdots\\
0 & 0 & 0       & \hat{q}^3     & 1       & \cdots\\
\vdots & \vdots & \vdots & \vdots & \vdots & \ddots
\end{array}\right).
\end{displaymath}
for some $c_{i,j}\in\Z_p$.

\begin{thm}
There exists an invertible upper triangular matrix $U$ such that $UCU^{-1}=R$, where
    \begin{displaymath}
        R=\left(\begin{array}{cccccc}
            1 & 1 & 0 & 0 & 0 & \cdots\\
            0 & \hat{q} & 1 & 0 & 0 & \cdots\\
            0 & 0 & \hat{q}^2 & 1 & 0 & \cdots\\
            0 & 0 & 0 & \hat{q}^3 & 1 & \cdots\\
            \vdots & \vdots & \vdots & \vdots & \vdots & \ddots
        \end{array}\right).
        \end{displaymath} 
One such is given by $U=(U_{i,j})_{i,j\geqslant0}\in U_\infty\Z_p$ where the first row is chosen to be
    \begin{displaymath}
        U_{0,j}=\left\{\begin{array}{ll}
        1 & \textrm{ if } j=0,\\
        0 & \textrm{ otherwise,}
    \end{array}\right.
\end{displaymath}
and the next row is defined recursively from the previous one by    
    $$
    U_{i+1,j}=\left(\sum_{s=i}^{j-2}U_{i,s}c_{s,j}\right)+U_{i,j-1}+(\hat{q}^j-\hat{q}^i)U_{i,j}.
    $$
\end{thm}

\begin{proof}
Let $U$ be the matrix defined recursively above. First we check that $U$ is upper triangular and invertible. It is clear that $U_{i,j}\in\Z_p$ for $i,j\geqslant0$. It can be shown that $U_{i,j}=0$ if $i>j$ by induction on $i$. It is true from the formula that $U_{1,0}=0$. Now assume that $U_{i-1,j}=0$ for all $j<i-1$. Then, for $i>j$, 
    $$
    U_{i,j}=U_{i-1,j-1}+(\hat{q}^j-\hat{q}^{i-1})U_{i-1,j}.
    $$ 
If $j<i-1$, then both $U_{i-1,j-1}$ and $U_{i-1,j}$ are zero by assumption. And if $j=i-1$, then $U_{i-1,j-1}$ is zero and $\hat{q}^{i-1}-\hat{q}^{i-1}=0$, so the induction is complete.

Now we show that $U_{i,i}\in\Z_p^\times$ for all $i\geqslant0$, so that $U$ is invertible.
Again we proceed by induction. Clearly $U_{0,0}=1\in\Z_p^\times$. Now assume that $U_{i,i}\in\Z_p^\times$. 
We have 
    $$
    U_{i+1,i+1}=U_{i,i}+(\hat{q}^{i+1}-\hat{q}^i)U_{i,i+1}
    $$ 
and since $\hat{q}^{i+1}-\hat{q}^i=\hat{q}^i(\hat{q}-1)\equiv0\mod p$, it follows that $U_{i+1,i+1}\in \Z_p^\times$. 

To show that $UCU^{-1}=R$, we compare entries $(UC)_{i,j}$ and $(RU)_{i,j}$. Diagonally 
$(UC)_{i,i}=\hat{q}^iU_{i,i}=(RU)_{i,i}$. Now let $j>i$; the entries of $UC$ and $RU$ are given by
    \begin{align*}
        (UC)_{i,j}&=\left(\sum_{s=i}^{j-2}U_{i,s}c_{s,j}\right)+U_{i,j-1}+\hat{q}^jU_{i,j},\\
        (RU)_{i,j}&=\hat{q}^iU_{i,j}+U_{i+1,j}.
    \end{align*}
Then the recurrence relation for the entries $U_{i,j}$ tells us that   
    $(UC)_{i,j}=(RU)_{i,j}$.
    
Hence $(UC)_{i,j}=(RU)_{i,j}$ for all $i,j\geqslant0$ and $j\geqslant i$.
\end{proof}

So we now have the following result.

\begin{thm}\label{isomatrix}
There is an isomorphism of groups 
    $$
    \Lambda': U_\infty\Z_p\rightarrow\Aut^0_{\text{left-}\ell\text{-mod}}(\ell\wedge \ell),
    $$
under which
the automorphism $1\wedge\Psi^q$ corresponds to the matrix
\begin{displaymath}
R=\left(\begin{array}{cccccc}
1 & 1 & 0 & 0 & 0 & \cdots\\
0 & \hat{q} & 1 & 0 & 0 & \cdots\\
0 & 0 & \hat{q}^2 & 1 & 0 & \cdots\\
0 & 0 & 0 & \hat{q}^3 & 1 & \cdots\\
\vdots & \vdots & \vdots & \vdots & \vdots & \ddots
\end{array}\right).
\end{displaymath}
The isomorphism is given by $\Lambda'(X)= \Lambda(B^{-1}XB)$, where
 $B=UE$ and $E$ and $U$ are the matrices above.
\end{thm}

\begin{proof}
Since $B$ is an invertible upper triangular matrix, $X\mapsto B^{-1}XB$ is a group isomorphism
$U_\infty\Z_p\rightarrow U_\infty\Z_p$ and the result follows.
\end{proof}

\begin{remark}
It would be interesting to find an explicit basis for which the matrix of $1\wedge\Psi^q$ is precisely
$R$. We hope to return to this in future work.
\end{remark}

\section{Applications}
\label{SecApplications}

In this section we present two applications. Firstly, we obtain precise information about the important map                           
    $$
    1\wedge\phi_n=1\wedge(\Psi^q-1)(\Psi^q-\hat{q})\ldots(\Psi^q-\hat{q}^{n-1}):\ell\wedge\ell\rightarrow\ell\wedge\ell.
    $$
We give closed formulas involving $q$-binomial coefficients for all the entries in the corresponding matrix.
Secondly, we give a new description of the ring $l^0(l)$ of degree zero stable operations for the ($p$-local) 
Adams summand in terms of matrices.

\subsection{The Map $1\wedge\phi_n$ and the Matrix $X_n$}

We apply the preceding result to study of the map                           
    $$
    1\wedge\phi_n=1\wedge(\Psi^q-1)(\Psi^q-\hat{q})\ldots(\Psi^q-\hat{q}^{n-1}):\ell\wedge\ell\rightarrow\ell\wedge\ell.
    $$
The analogous map was first studied by Milgram in~\cite{Milgram} in relation to real 
connective $K$-theory $ko$ localised at the prime $2$. We follow the method used in~\cite[Theorem $5.4$]{BaSn}, but we are
able to produce new closed formulas for every entry in the matrix corresponding to the above map, in terms of
$q$-binomial coefficients (also known as Gaussian polynomials). A short discussion of the relevant information about these can be found in the
appendix.

We will write $\tilde{U}_\infty\Z_p$ for the ring of upper triangular matrices with entries in the $p$-adic integers.
The group $U_\infty\Z_p$ is a subgroup of the multiplicative group of units in this ring. Generalising the group isomorphism $\Lambda'$ of 
Theorem~\ref{isomatrix} we can construct the following diagram
    \begin{displaymath}
        \xymatrix{
            U_\infty\Z_p \ar[rrr]^{\Lambda'}_{\cong} \ar[d]_\cap  &&& \Aut^0_{\text{left-}\ell\text{-mod}}(\ell\wedge \ell) \ar[d]_\cap \\
            \tilde{U}_\infty\Z_p \ar[rrr]_{\lambda'} &&& \End_{\text{left-}\ell\text{-mod}}(\ell\wedge \ell)}
    \end{displaymath}
where $\lambda'_{|U_\infty\Z_p}=\Lambda'$. Recall that the map $\Lambda'$ sends a matrix 
$A\in U_\infty\Z_p$ to $\Lambda'(A)=\Lambda(B^{-1}AB)=\sum_{m\geqslant n}(B^{-1}AB)_{n,m}\iota_{m,n}$. We extend this by letting 
    $$
    \lambda(A')=\sum_{m\geqslant n}(B^{-1}A'B)_{n,m}\iota_{m,n}
    $$ 
to obtain a left-$\ell$-module endomorphism of $\ell\wedge\ell$. This is a multiplicative map by the same argument given for $\Lambda$ in the proof of Proposition~\ref{extprod}.

By moving from $U_\infty\Z_p$ to $\tilde{U}_\infty\Z_p$ it is now possible to use the additive structure given by matrix addition
and it is easy to check that $\lambda'$ respects addition.

\begin{dfn}
Let 
    $$
    \phi_n=(\Psi^q-1)(\Psi^q-\hat{q})\cdots(\Psi^q-\hat{q}^{n-1})
    $$ and let $R_n=R-\hat{q}^{n-1}I\in\tilde{U}_\infty\Z_p$ and $X_n=R_1R_2\cdots R_n\in\tilde{U}_\infty\Z_p$.
Here $I$ denotes the infinite identity matrix.
\end{dfn}

By Theorem~\ref{isomatrix}, the map $1\wedge\Psi^q$ corresponds to the matrix $R$.
It follows that $1\wedge\phi_n$ corresponds to the matrix $X_n$.

A basic tool we will use is splitting up the matrix $R$ 
in order to easily calculate its powers.

\begin{dfn}\label{SD}
Define matrices $D$ and $S$ in $\tilde{U}_\infty\Z_p$ by
    \begin{displaymath}
    D=\left(\begin{array}{cccccc}
    1 & 0 & 0 & 0 & 0 & \cdots\\
    0 & \hat{q} & 0 & 0 & 0 & \cdots\\
    0 & 0 & \hat{q}^2 & 0 & 0 & \cdots\\
    0 & 0 & 0 & \hat{q}^3 & 0 & \cdots\\
    \vdots & \vdots & \vdots & \vdots & \vdots & \ddots
    \end{array}\right),\qquad
    S=\left(\begin{array}{cccccc}
    0 & 1 & 0 & 0 & 0 & \cdots\\
    0 & 0 & 1 & 0 & 0 & \cdots\\
    0 & 0 & 0 & 1 & 0 & \cdots\\
    0 & 0 & 0 & 0 & 1 & \cdots\\
    \vdots & \vdots & \vdots & \vdots & \vdots & \ddots
    \end{array}\right).
\end{displaymath}
\end{dfn}

Then $R=D+S$. Since powers of $D$ and $S$ are easy to calculate,
and we have $SD=\hat{q}DS$,  we are in a situation where we can apply the $q$-binomial theorem
to calculate $R^n=(D+S)^n$. See the appendix for a short discussion of the
$q$-binomial coefficients ${n\brack m}_{q}$.

\begin{lem}\label{Rpower}
    $$
    (R^n)_{s,s+c}={n\brack n-c}_{\hat{q}}\hat{q}^{(s-1)(n-c)}.
    $$
\end{lem}

\begin{proof}
We note that $D^iS^j$ is given by
    \begin{displaymath}
    (D^iS^j)_{s,t}=\begin{cases}
    \hat{q}^{(s-1)i}&\textrm{ if }t=s+j,\\
    0&\textrm{ otherwise.}
    \end{cases}
    \end{displaymath}
Applying the $q$-binomial theorem~(\ref{expand}), we have
    $$
    (R^n)_{s,s+c}=((D+S)^n)_{s,s+c}=\sum_{i=0}^n{n\brack i}_{\hat{q}}(D^iS^{n-i})_{s,s+c}.
    $$
\noindent For any particular value of $c$ at most one term in this sum is non-zero, namely the $i=n-c$ term 
if $0\leqslant c\leqslant n$. The result follows.
\end{proof}

Let $\Omega$ denote the homotopy equivalence giving Kane's splitting,    
    $$
    \Omega:\bigvee_{n\geq0}\ell\wedge\K(n)\rightarrow\ell\wedge\ell.
    $$

\begin{thm}\label{app}
\begin{list}
{(\arabic{itemcounter})}
{\usecounter{itemcounter}\leftmargin=0.5em}
\item The first $n$ columns of the matrix $X_n$ are trivial.
\item Let $C_n$ be the mapping cone of the restriction of $\Omega$ to the first $n$ pieces of the splitting of $\ell\wedge\ell$, i.e.     
    $$
    C_n=\text{Cone}\left(\Omega_|:\bigvee_{0\leqslant m\leqslant n-1}\ell\wedge\K(m)\rightarrow\ell\wedge\ell\right).
    $$ 
Then in the $p$-complete stable homotopy category there exists a commutative diagram of left $\ell$-module spectra of the form
    \begin{displaymath}
    \xymatrix{\ell\wedge\ell \ar[rr]^{1\wedge\phi_n} \ar[dr]_{\pi_n} && \ell\wedge\ell\\
    & C_n \ar[ur]_{\hat{\phi}_n}&    }
    \end{displaymath}
    where $\pi_n$ is the cofibre of $\Omega_|$ and $\hat{\phi}_n$ is determined up to homotopy by the diagram.
\item For $n\geqslant1$, we have $(X_n)_{s,s+c}=0$ if $c<0$ or $c>n$ and for $0\leqslant c\leqslant n$,
    $$
    (X_n)_{s,s+c}=
                \sum_{i=c}^n(-1)^{n-i}\hat{q}^{{n-i\choose 2}+(s-1)(i-c)}{n\brack i}_{\hat{q}}{i\brack i-c}_{\hat{q}}.
    $$
\end{list}
\end{thm}

\begin{proof}
\quad (1)\ \ The result is certainly true of $X_1=R_1$. We prove the result for all $n\geqslant1$ by induction. 
Assume that the first $n$ columns of $X_n$ are trivial, i.e. $(X_n)_{i,j}=0$ if $j\leqslant n$. By definition $X_{n+1}=X_nR_{n+1}$. 
We also have $(R_{n+1})_{i,j}=0$ unless $(i,j)=(s,s)$ or $(s,s+1)$ and $(R_{n+1})_{n+1,n+1}=0$. Now 
    $$
    (X_{n+1})_{i,j}=(X_n)_{i,j-1}(R_{n+1})_{j-1,j}+(X_n)_{i,j}(R_{n+1})_{j,j}.
    $$ 
This is zero if $j\leqslant n$ because $(X_n)_{i,j-1},(X_n)_{i,j}=0$ and it is zero if $j=n+1$ 
because $(X_n)_{i,n},(R_{n+1})_{n+1,n+1}=0$.
\smallskip

\noindent (2)\ \  We know that $1\wedge\phi_n=\lambda'(X_n)$. In order for $1\wedge\phi_n$ to factor via $C_n$ (and for the diagram to commute) we need to show that $X_n$ corresponds under $\lambda'$ to a left $\ell$-module endomorphism of $\vee_{m\geqslant0}\ell\wedge\K(m)$ which is trivial on each piece $\ell\wedge\K(m)$ where $m\leqslant n-1$. The map $\lambda'(X_n)$ acts trivially on pieces $\ell\wedge\K(m)$ where $m\leqslant n-1$ if each map $\iota_{m,k}:\ell\wedge\K(m)\rightarrow\ell\wedge\K(k)$ has coefficient zero when $m\leqslant n-1$ in the explicit description of $\lambda'(X_n)$. This corresponds to the condition $(X_n)_{k,m}=0$ when $m\leqslant n-1$, which is true by part (1).
\smallskip 

\noindent (3)\ \ Certainly $X_n$ is upper triangular, so $(X_n)_{s,s+c}=0$ if $c<0$. We show that $(X_n)_{s,s+c}=0$ if $c>n$ by induction on $n$. The initial case for the induction is $X_1$ where this clearly holds. Assume that $(X_{n-1})_{s,s+c}=0$ if $c>n-1$. As in part (1), we have $X_n=X_{n-1}R_n$ and  $(X_n)_{i,j}=(X_{n-1})_{i,j-1}(R_n)_{j-1,j}+(X_{n-1})_{i,j}(R_n)_{j,j}$. Now let $j>n$, then     
    $$
    (X_n)_{s,s+j}=(X_{n-1})_{s,s+j-1}(R_n)_{s+j-1,s+j}+(X_{n-1})_{s,s+j}(R_n)_{s+j,s+j}
    $$ 
and this is zero because both $(X_{n-1})_{s,s+j-1}$ and $(X_{n-1})_{s,s+j}$ are zero by the inductive hypothesis.

For the second part, by~\cite[Proposition 8]{ccw},
    $$
    X_n=\sum_{i=0}^n(-1)^{n-i}\hat{q}^{n-i\choose 2}{n\brack i}_{\hat{q}}R^i.
    $$
Hence, using Lemma~\ref{Rpower},
    \begin{align*}
    (X_n)_{s,s+c}&=\sum_{i=0}^n(-1)^{n-i}\hat{q}^{n-i\choose 2}{n\brack i}_{\hat{q}}(R^i)_{s,s+c}\\
    &=\sum_{i=c}^n(-1)^{n-i}\hat{q}^{n-i\choose 2}{n\brack i}_{\hat{q}}{i\brack i-c}_{\hat{q}}\hat{q}^{(s-1)(i-c)}.
    \end{align*}
    The range of the final sum can be restricted to starting from $c$ rather than $0$ as the second $q$-binomial coefficient is zero
    for $i\leqslant c$.
\end{proof}

\subsection{$K$-Theory Operations}

The matrix approach provides another way of viewing the ring of stable
degree zero operations on the $p$-local Adams summand. We will work in this final section in the $p$-local stable homotopy category. In a slight abuse of notation let $\ell$ now denote the Adams summand of $p$-local complex connective $K$-theory (rather than the $p$-complete version). 
Let $\tilde{U}_\infty\Z_{(p)}$ be the ring of upper triangular matrices with entries in the $p$-local integers.

\begin{dfn}
We define a filtration on $\tilde{U}_\infty\Z_{(p)}$ by, for $n\in\N$, 
    $$
    U_n=\{X\in\tilde{U}_\infty\Z_{(p)}:x_{i,j}=0\text{ if }j\leqslant n\}.
    $$ 
This gives a decreasing filtration 
    $$
    \tilde{U}_\infty\Z_{(p)}=U_0\supset U_1\supset U_2\supset\cdots
    $$ 
where each $U_n$ is a two-sided ideal of $\tilde{U}_\infty\Z_{(p)}$.
\end{dfn}

This column filtration gives
a filtration by two-sided ideals because the matrices are upper triangular
(and this would not be the case if we filtered by rows). This can be regarded as the natural filtration 
on $\tilde{U}_\infty\Z_{(p)}$ and $\tilde{U}_\infty\Z_{(p)}$ is complete with respect to this topology.

\begin{thm}\label{topringapp}
The ring of degree zero stable operations of the Adams summand, $\ell^0(\ell)$, is isomorphic as a topological ring 
to the completion of the subring of $\tilde{U}_\infty\Z_{(p)}$ generated by the matrix $R$.
\end{thm}

\begin{proof}
We have the following description of $\ell^0(\ell)$ from~\cite[Theorem 4.4]{CCW2} 
    $$
    \ell^0(\ell)=\left\{\sum_{n=0}^\infty a_n\phi_n:a_n\in\Z_{(p)}\right\}.
    $$ 
This is complete in the filtration topology when filtered by the
ideals 
    $$
    \left\{\sum_{n=m}^\infty a_n\phi_n:a_n\in\Z_{(p)}\right\}.
    $$
Define a map $\alpha:\ell^0(\ell)\rightarrow\tilde{U}_\infty\Z_{(p)}$ as
 the continuous ring homomorphism determined by $\alpha(\Psi^q)=R$. 
We have 
    $$
    \alpha(\phi_n)=\alpha(\prod_{i=0}^{n-1}(\Psi^q-\hat{q}^i))=\prod_{i=0}^{n-1}(R-\hat{q}^i)=X_n.
    $$
By Theorem~\ref{app} (1), the first $n$ columns of $X_n$ are trivial, so $\alpha(\phi_n)\in U_n$. Thus $\alpha$ respects the filtration and so when applied to infinite sums $\alpha\left(\sum_{n=0}^\infty a_n\phi_n\right)=\sum_{n=0}^\infty a_nX_n$ is well-defined (each entry in the matrix is a finite sum).

We have $\Ker\alpha=\left\{\sum_{n=0}^\infty a_n\phi_n:a_n=0\text{ for all }n\right\}=0$, so $\alpha$ is injective.
Let $S=\left\{\sum_{n=0}^N a_n R^n:a_n\in\Z_{(p)},N\in\N_0\right\}$. It is clear that $S\subseteq\im(\alpha)$. Because $\alpha$ is continuous and $\tilde{U}_\infty\Z_{(p)}$ is complete it follows that the completion of $S$ is precisely the image of $\alpha$. 
\end{proof}

Similar descriptions can be given for ${ku_{(p)}}^0(ku_{(p)})$
and $ko_{(2)}^0(ko_{(2)})$.

\section{Appendix: The $q$-Binomial Theorem}
\label{appendix}

The $q$-binomial coefficients, also known as Gaussian polynomials, arise in many diverse areas of mathematics. 
Perhaps the nicest way to define them is as the coefficients arising in the
following version of the
$q$-binomial theorem. If $X$ and $Y$ are variables which $q$-commute,
that is, $YX=qXY$, then for $n\in\N_0$ we have
    \begin{equation}\label{expand}
        (X+Y)^n=\sum_{i=0}^n{n\brack i}_q X^iY^{n-i},
    \end{equation}
where ${n\brack i}_q$ is a $q$-binomial coefficient.
This version of the $q$-binomial theorem goes back to~\cite{Schutzenberger}.
(Various other results also go under the name of $q$-binomial theorem.)

The above point of view has several nice features. It makes evident the relationship with the
ordinary binomial coefficients and that the
coefficients ${n\brack i}_q$ are indeed integer polynomials in $q$. The 
two standard recurrences for $q$-binomial coefficients are easily read off, by writing
$(X+Y)^n$ as $(X+Y)(X+Y)^{n-1}$ and as $(X+Y)^{n-1}(X+Y)$.

On the other hand, if one starts from the closed formula
for the $q$-binomial coefficients:
    $$
    {n\brack i}_q=\prod_{j=0}^{i-1}\frac{1-q^{n-j}}{1-q^{i-j}}\qquad\text{where $n,i\in\N_0$},
    $$ 
then it is easy to deduce the standard recurrences and~(\ref{expand}) can be readily proved
via induction and either one of them.

\end{document}